\documentclass[review]{elsarticle}

\usepackage{amssymb}
\usepackage{amsfonts}
\usepackage{amsmath}
\usepackage{mathtools}
\usepackage{amsthm}

\usepackage[letterpaper,margin=1in]{geometry}

\usepackage{enumerate}

      \theoremstyle{plain}
      \newtheorem{theorem}{Theorem}

      \newtheorem{proposition}[theorem]{Proposition}
      
      \newtheorem{question}[theorem]{Question}

      \theoremstyle{definition}
      \newtheorem{definition}[theorem]{Definition}
      \newtheorem{notation}[theorem]{Notation}

      \theoremstyle{remark}

      \theoremstyle{plain}
      \newtheorem*{theorem*}{Theorem}
      \newtheorem*{lemma*}{Lemma}
      \newtheorem*{corollary*}{Corollary}
      \newtheorem*{proposition*}{Proposition}
      \newtheorem*{conjecture*}{Conjecture}
      \newtheorem*{question*}{Question}
      \newtheorem*{claim*}{Claim}

      \theoremstyle{definition}
      \newtheorem*{definition*}{Definition}
      \newtheorem*{example*}{Example}
      \newtheorem*{observation*}{Observation}
      \newtheorem*{game*}{Game}

      \theoremstyle{remark}
      \newtheorem*{remark*}{Remark}

\usepackage{clontzDefinitions}

\begin{document}

\title{Limited Information Strategies and Discrete Selectivity}
\author[SC]{Steven Clontz}
\ead{sclontz@southalabama.edu}
\author[JH]{Jared Holshouser}
\ead{JaredHolshouser@southalabama.edu}
\address[SC]{Department of Mathematics and Statistics,
The University of South Alabama,
Mobile, AL 36688}
\address[JH]{Department of Mathematics and Statistics,
The University of South Alabama,
Mobile, AL 36688}


\begin{abstract}
 We relate the property of discrete selectivity and its corresponding game, both recently introduced by V.V. Tkachuck, to a variety of selection principles and point picking games. In particular we show that player II can win the discrete selection game on \(C_p(X)\) if and only if player II can win a variant of the point open game on \(X\). We also show that the existence of limited information strategies in the discrete selection game on \(C_p(X)\) for either player are equivalent to other well-known topological properties.
\end{abstract}

\maketitle

\renewcommand{\rothGame}[1]{\ensuremath{G_1(\mc O_{#1},\mc O_{#1})}}

\section{Introduction}

In the course of studying the strong domination of function spaces by second countable spaces and countable spaces, G. Sanchez and Tkachuk isolated the topological property of discrete selectivity\cite{SanchezTkachuk}\cite{Tkachuk1}.
A space is discretely selective if for every sequence \(\{U_n : n \in \omega\}\) of non-empty open subsets of the space, there are points \(x_n \in U_n\) so that \(\{x_n : n \in \omega\}\) is closed discrete.
In subsequent work, Tkachuk showed that for \(T_{3.5}\)-spaces, \(C_p(X)\) is discretely selective if and only if \(X\) is uncountable.

Discrete selectivity naturally generates a game, in which player I plays open sets, player II responds with points from those open sets, and player II wins if the points form a closed discrete set.
Tkachuk explored what happens when player I has a winning strategy for this game, showing that the existence of a winning strategy for player I in this game on \(C_p(X)\) is equivalent to player I having a winning strategy for Gruenhage's \(W\)-game on \(C_p(X,\mathbf 0)\) and is also equivalent to player I having a winning strategy for the point-open game on \(X\)\cite{Tkachuk3}.
Tkachuk also showed that if player II has a winning strategy in the point-open game on \(X\), then player II has a winning strategy in the discrete selection game on \(C_p(X)\).
Tkachuk hypothesized that the implication partially reverses for player II
(considering \(\omega\)-covers), and posed this problem as an open question.
All of the strategies Tkachuk worked with were perfect information strategies.

By considering limited information strategies and other topological games, we were able to answer Tkachuk's question and uncover a number of interesting connections between the discrete selection game and other topological properties.
Classic works by Telgarksy and Galvin show that the point open game is dual to the Rothberger game\cite{Galvin}.
Clontz, in work prior to this, established the equivalence of the existence of winning strategies for the Rothberger game and variants of the Rothberger game on \(X\) to the existence of winning strategies in games related to countable fan tightness for \(C_p(X)\)\cite{Clontz1}.
Clontz did this both for strategies of perfect information and for limited information strategies.
Starting with these results, we were able to relate a host of games on \(C_p(X)\) and \(X\) for strategies of both limited information and perfect information.
As a result we answer Tkachuk's question: player II has a winning strategy for the discrete selection game on \(C_p(X)\) if and only if player II has a winning strategy for the \(\omega\)-cover variant of the finite-open game on \(X\).
The \(\omega\)-cover variant of the finite-open game is closely related to the point open game, but it is consistent that they are different.
Tkachuk referred to a strategy for this variant for player II as an almost winning strategy.
So in Tkachuk's terminology, player II has a winning strategy for the discrete selection game on \(C_p(X)\) if and only if player II has an almost winning strategy for the point-open game on \(X\).
Moreover, we answered the implied question ``what topological property does a winning strategy for player II for the discrete selection game on \(C_p(X)\) correspond to?''
We show that player II has a winning strategy for the discrete selection game on \(C_p(X)\) if and only if \(X\) is not Rothberger with respect to \(\omega\)-covers.
This in turn is true if and only if some finite power of \(X\) is not Rothberger.

\section{Definitions}

We will be using a number of definitions.
These are broken up into three main categories: labeling schema, topological notions, and games.
\(\omega=\{0,1,2,\dots\}\) refers to the natural numbers, 
\(A^{<\omega}\) collects all the finite tuples with entries from \(A\),
and \([A]^{<\omega}\) collects all the finite subests of \(A\).

\subsection{Labeling Schema}

\begin{definition}
 The \term{selection principle} \(\schSelProp{\mc A}{\mc B}\) states that given \(A_n\in\mc A\) for \(n<\omega\), there exist \(B_n\in[A_n]^{<\omega}\) such that \(\bigcup_{n<\omega}B_n\in\mc B\).
\end{definition}

\begin{definition}
 An \term{\(\omega\)-length game} \(G=\<M,W\>\) is played by two players \(\plI\) and \(\plII\).
 Each round, the players alternate choosing moves \(a_n\) and \(b_n\) from the moveset \(M\).
 If the seqeunce \(\<a_0,b_0,a_1,b_1,\dots\>\) belongs to the payoff set \(W\), then \(\plI\)
 is the winner; otherwise \(\plII\) is the winner.
 
 A \term{strategy} is a function \(\sigma:M^{<\omega}\to M\) which is used to decide the
 move for a particular player. For \(\plI\), \(\sigma(\emptyset)\) is the first move,
 and if \(\plII\) responds with \(b_0\), then \(\sigma(\<b_0\>)\) yields \(\plI\)'s next move,
 and so on. Likewise, the first two moves for \(\plII\) using a strategy \(\sigma\)
 would be \(\sigma(\<a_0\>)\) and \(\sigma(\<a_0,a_1\>\) in response to \(\plI\)'s moves
 \(a_0\) and \(a_1\).
 
 A strategy is said to be a \term{winning strategy} for a player if it always 
 guarantees a victory for that player, regardless of the moves chosen by the opponent in response.
 If \(\plI\) has a winning strategy for \(G\), we write \(\plI\win G\); likewise we write
 \(\plII\win G\) if \(\plII\) has a winning strategy for \(G\).
 Of course, both players cannot have winning strategies for the same game (although there
 do exist \term{indetermined} games for which \(\plI\notwin G\) and \(\plII\notwin G\)).
\end{definition}

\begin{definition}
 The \term{selection game} \(\schSelGame{\mc A}{\mc B}\) is the analogous game to 
 \(\schSelProp{\mc A}{\mc B}\), where during each round \(n<\omega\), Player \(\plI\) 
 first chooses \(A_n\in\mc A\), and then Player \(\plII\) chooses \(B_n\in[A_n]^{<\omega}\).
 Player \(\plII\) wins in the case that \(\bigcup_{n<\omega}B_n\in\mc B\), 
 and Player \(\plI\) wins otherwise.

 A strategy for \(\plII\) in the game \(\schSelGame{\mc A}{\mc B}\) is then a function 
 \(\sigma\) satisfying \(\sigma(\<A_0,\dots,A_n\>)\in[A_n]^{<\omega}\) for 
 \(\<A_0\,\dots,A_n\>\in\mc A^{n+1}\), and is winning if whenever \(\plI\) plays 
 \(A_n\in\mc A\) during each round \(n<\omega\), \(\plII\) wins the game by 
 playing \(\sigma(\<A_0,\dots,A_n\>)\) during each round \(n<\omega\).
\end{definition}

\begin{definition}
 In addition to strategies which have access to all the previous moves of the game (also known as perfect information), we will consider the existence of strategies which use less information.
 A \term{Markov strategy} is a strategy which tells the player what to play given only the most recent move of the opponent and the current round number.
 For \(\plI\), it is a function \(\sigma(Y,n)\), where \(Y\) is a possible play from \(\plII\) and \(n \in \omega\).
 If \(n = 0\), \(Y\) is taken to be \(\emptyset\).
 If \(\plI\) has a winning Markov strategy, we write \(\plI\markwin G\).
 For \(\plII\) it is a function \(\sigma(X,n)\), where \(X\) is a possible play from \(\plI\) and \(n \in \omega\).
 If \(\plII\) has a winning Markov strategy, we write \(\plII\markwin G\).

 More specifically, A \term{Markov strategy} for \(\plII\) in the game \(\schSelGame{\mc A}{\mc B}\) is a function \(\sigma\) satisfying \(\sigma(A,n)\in[A_n]^{<\omega}\) for \(A\in\mc A\) and \(n<\omega\). We say this Markov strategy is \term{winning} if whenever \(\plI\) plays \(A_n\in\mc A\) during each round \(n<\omega\), \(\plII\) wins the game by playing \(\sigma(A_n,n)\) during each round \(n<\omega\).
  
 A \term{tactic} is a strategy which only depends on the most recent play of the opponent.
 If \(\plI\) has a winning tactic, we write \(\plI\tactwin G\) and if \(\plII\) has a winning tactic, we write \(\plII\tactwin G\).
 In some instances, player I will be able to win a game regardless of what II is playing.
 In this case, it is possible to have a strategy for I which depends only on the round of the game.
 We say I has a \term{predetermined strategy} and write \(\plI\prewin G\). 
\end{definition}

\begin{notation}
 If \(\schSelProp{\mc A}{\mc B}\) characterizes the property \(P\), then we say \(\plII\win\schSelGame{\mc A}{\mc B}\) characterizes \(P^+\) (``strategically \(P\)''), and \(\plII\markwin\schSelGame{\mc A}{\mc B}\) characterizes \(P^{\plusMark}\) (``Markov \(P\)'').
 Of course, \(P^{\plusMark}\Rightarrow P^+ \Rightarrow P\).
\end{notation}

\begin{definition}
 Let \(\schStrongSelProp{\mc A}{\mc B},\schStrongSelGame{\mc A}{\mc B}\) be the natural variants of \(\schSelProp{\mc A}{\mc B},\schSelGame{\mc A}{\mc B}\) where each choice by \(\plII\) must either be a single element or singleton (whichever is more convenient for the proof at hand), rather than a finite set.
 Convention calls for denoting these as \term{strong} versions of the corresponding selection principles and games, denoted here as \(sP\) for property \(P\), with a few exceptions for properties which already have their own names.
\end{definition}

\begin{definition}
 We will use the following shorthand for various special collections of subsets of \(X\).
  \begin{itemize}
   \item Let \(\mc O_X\) be the collection of open covers for a topological space   \(X\).
   \item An \term{\(\omega\)-cover} \(\mc U\)   for a topological space \(X\) is an open cover   such that for every \(F\in[X]^{<\omega}\), there exists some \(U\in\mc U\)   such that \(F\subseteq U\). Let \(\Omega_X\) be the collection of \(\omega\)-covers for a topological   space \(X\).
   \item Let \(\Omega_{X,x}\) be the collection of subsets \(A\subset X\) where   \(x\in \overline{A}\). (Call \(A\) a \term{blade} of \(x\).)
   \item Let \(\mc D_{X}\) be the collection of dense subsets of a topological   space \(X\).
   \item Let \(T_X\) to be the non-empty open subsets of \(X\).
   \item Let \(T_{X,x} = \{U \in T_X : x \in U\}\).
  \end{itemize}
\end{definition}

\subsection{Topological Notions}

\begin{definition}\label{SelectionPrinciples}
 Using the notation just established, we can record a number of topological properties.
  \begin{itemize}
   \item \(\schSelProp{\mc O_X}{\mc O_X}\) is the well-known \term{Menger property} for \(X\) (\(M\) for short).
     \begin{itemize}
       \item \(\schStrongSelProp{\mc O_x}{\mc O_X}\) is the well-known \term{Rothberger property}
       (\(R\) for short), so we say this instead of strong Menger or \(sM\).
     \end{itemize}
   \item \(\schSelProp{\Omega_X}{\Omega_X}\) is the \term{\(\Omega\)-Menger property} for \(X\) (\(\Omega M\) for short).
     \begin{itemize}
       \item Likewise we call \(\schStrongSelProp{\Omega_X}{\Omega_X}\) the 
       \term{\(\Omega\)-Rothberger property} for \(X\) (\(\Omega R\) for short). 
     \end{itemize}
   \item \(\schSelProp{\Omega_{X,x}}{\Omega_{X,x}}\) is the \term{countable fan tightness property} for \(X\) at \(x\) (\(CFT_x\) for short). A space \(X\) has \term{countable fan tightness} (\(CFT\) for short)
    if it has countable fan tightness at each point \(x\in X\).
   \item \(\schSelProp{\mc D_X}{\Omega_{X,x}}\) is the \term{countable dense fan tightness property} for \(X\) at \(x\) (\(CDFT_x\) for short). A space \(X\) has \term{countable dense fan tightness} (\(CDFT\) for short) if it has countable dense fan tightness at each point \(x\in X\).
  \end{itemize}
\end{definition}

Note that for homogeneous spaces such as \(C_p(X)\), \(C(D)FT_x\) is equivalent to \(C(D)FT\).

Tkachuk isolated the following notion in \cite{Tkachuk2}.

\begin{definition}
 A space \(X\) is \term{discretely selective} if whenever \(\{U_n : n \in \omega\}\) is a sequence of open subsets of \(X\), there are points \(x_n \in U_n\) so that \(\{x_n : n \in \omega\}\) is closed discrete.
\end{definition}

We will use the following notation when working with \(C_p(X)\).

\begin{definition}
 Suppose \(X\) is \(T_{3.5}\).
 Basic open subsets of \(C_p(X)\) will be written as
 \[
  [f,F,\epsilon] = \{g\in C_p(X):|g(x)-f(x)|<\epsilon\text{ for all }x\in F\}
 \]
 where \(f \in C_p(X)\), \(F\) is a finite subset of \(X\), and \(\epsilon > 0\) is a real number.
 \(F\) is called the \term{support} of \([f,F,\epsilon]\).
 It follows that all open \(U \subseteq C_p(X)\) restrict only finitely many
 coordinates, which we label \(\supp(U)\).
\end{definition}

\subsection{Topological Games}

\begin{definition}
 Selection games associated with the principles listed in Definition \ref{SelectionPrinciples} 
 will be investigated in this paper; for example, \(\schStrongSelGame{\mc O_X}{\mc O_X}\) is the well-known 
 \term{Rothberger game}.
\end{definition}
  
\begin{definition}
 The following point-picking games will also be played in this paper.
  \begin{itemize}
   \item The \term{point-open game} for \(X\), denoted \(PO(X)\), is played as follows. Each round, player I plays a point \(x_n \in X\) and player II plays an open sets \(U_n\) with the property that \(x_n \in U_n\). I wins the play of the game if \(X = \bigcup_n U_n\).
\begin{itemize}
   \item The \term{finite-open game} for \(X\), denoted \(FO(x)\), is played similarly, except that I now plays finite subsets of \(X\), and II's open sets must cover I's corresponding finite sets.
   \item \(\Omega FO(X)\) and \(\Omega PO(X)\) are defined similarly, but \(\plI\) now wins if \(\{U_n : n \in \omega\}\) forms an \(\omega\)-cover of \(X\).
\end{itemize}
   \item Fix \(x \in X\). \term{Gruenhage's \(W\)-game} for \(x\), denoted \(\gruConGame{X}{x}\), is played as follows. Each round, player I plays an open set \(U_n\) with the property that \(x \in U_n\) and player II plays a point \(x_n \in U_n\). I wins if \(x_n \to x\).
\begin{itemize}
   \item \term{Gruenhage's clustering-game} for \(x\), denoted \(\gruClusGame{X}{x}\), is played the same as \(\gruConGame{X}{x}\), except that I wins if \(x\) is a cluster point of \(\{x_n : n \in \omega\}\).
\end{itemize}
   \item Fix \(x \in X\). The \term{closure game} for \(x\), denoted \(CL(X,x)\), is played as follows. Each round, player I plays an open set \(U_n\) and II plays a point \(x_n \in U_n\). I wins if \(x \in \overline{\{x_n : n \in \omega\}}\).
\begin{itemize}
   \item The \term{discrete selectivity game}, denoted \(CD(X)\), is played the same as \(CL(X,x)\), but now II wins if \(\{x_n : n \in \omega\}\) is closed and discrete.
 \end{itemize}
\end{itemize}
\end{definition}

It's worth noting that selection principles may be characterized
using limited information strategies for seleciton games.

\begin{proposition}
  \(\schStrongSelProp{\mathcal{A}}{\mathcal{B}}\) if and only if 
  \(\plI \notprewin\schStrongSelGame{\mathcal{A}}{\mathcal{B}}\).
\end{proposition}
\begin{proof}
 First suppose that \(\schStrongSelProp{\mathcal{A}}{\mathcal{B}}\) holds.
 Let \(\sigma\) be a tentative predetermined strategy for \(\plI\) for \(\schStrongSelGame{\mathcal{A}}{\mathcal{B}}\).
 Then \(\{\sigma(n) : n \in \omega\} \subseteq \mathcal{A}\), and therefore there are \(B_n \in \sigma(n)\) for all \(n\) so that \(\bigcup_n B_n \in \mathcal{B}\).
 Thus \(\sigma\) is not a winning strategy for \(\plI\).
 So \(\plI \notprewin\schStrongSelGame{\mathcal{A}}{\mathcal{B}}\).
 
 Now suppose that \(\schStrongSelProp{\mathcal{A}}{\mathcal{B}}\) is false.
 Then there is some sequence \(\{A_n : n \in \omega\} \subseteq \mathcal{A}\) with the property that whenever \(B_n \in A_n\) for all \(n\), \(\bigcup_n B_n \notin \mathcal{B}\).
 Then the predetermined strategy \(\sigma(n) = A_n\) is winning for \(\plI\) for \(\schStrongSelGame{\mathcal{A}}{\mathcal{B}}\).
 Thus \(\plI \prewin\schStrongSelGame{\mathcal{A}}{\mathcal{B}}\).
\end{proof}

The proof of the following is similar.

\begin{proposition}
  \(\schSelProp{\mathcal{A}}{\mathcal{B}}\) if and only if 
  \(\plI \notprewin\schSelGame{\mathcal{A}}{\mathcal{B}}\).
\end{proposition}

\section{Strategies for Player I for the Discrete Selection Game on \(C_p(X)\)}

We begin by extending theorem 3.8 of Tkachuk\cite{Tkachuk3} to equate the existence of strategies for 11 games.

\begin{theorem}
 The following are equivalent for \(T_{3.5}\) spaces \(X\).
 \begin{enumerate}[a)]
  \item \(\plII\win \schStrongSelGame{\mc O_X}{\mc O_X}\), that is, \(X\) is \(R^+\).
  \item \(\plII\win \schStrongSelGame{\Omega_X}{\Omega_X}\), that is, \(X\) is \(\Omega R^+\).
  \item \(\plI\win PO(X)\). 
  \item \(\plI\win FO(X)\).
  \item \(\plI\win\Omega FO(x)\).
  \item \(\plI\win \gruConGame{C_p(X)}{\mathbf 0}\).
  \item \(\plI\win \gruClusGame{C_p(X)}{\mathbf 0}\).
  \item \(\plI\win CL(C_p(X),\mathbf 0)\).
  \item \(\plI\win CD(C_p(X))\).
  \item \(\plII\win \schStrongSelGame{\Omega_{C_p(X),\mathbf 0}}{\Omega_{C_p(X),\mathbf 0}}\),
        that is, \(C_p(X)\) is \(sCFT^+\).
  \item \(\plII\win \schStrongSelGame{\mc D_{C_p(X)}}{\Omega_{C_p(X),\mathbf 0}}\),
        that is, \(C_p(X)\) is \(sCDFT^+\).
 \end{enumerate}
\end{theorem}

\begin{proof}
 We will first show that (a) implies (b).
 So assume \(X\) is \(R^+\).
 In \cite{DiasScheepers}, it is shown that \(X^m\) is also \(R^+\) for all finite \(m\).
 Given an \(\omega\)-cover \(\mathcal{U}\), let \((\mathcal{U})^m = \{U^m : U \in \mathcal{U}\}\) and note that \((\mathcal{U})^m\) is an open cover \(X^m\).
 
 Now let \(\sigma_m\) be a winning strategy for \(\plII\) for the Rothberger game on \(X^m\).
 We define a strategy \(\sigma\) for \(\plII\) for \(\schStrongSelGame{\Omega_X}{\Omega_X}\) as follows.
 First let \(b:\omega \to \omega^2\) be a bijection, we will use this to layer the strategies together.
 At round \(n\), let \(m,k \in \omega\) be so that \(b(n) = (m,k)\).
 Suppose \(\plI\) has played \(\mathcal{U}_0,\cdots,\mathcal{U}_n\) up to this point.
 If \(\sigma_m((\mathcal{U}_0)^m,\cdots,(\mathcal{U}_n)^m) = (U_n)^k\), then \(\sigma(\mathcal{U}_0,\cdots,\mathcal{U}_n)\) is set to be \(U_n\).
 This completely defines the strategy \(\sigma\).
 
 Now suppose \(\tau\) is an attack by \(\plI\) against \(\sigma\).
 Say \(\plII\) played \(\{U_n : n \in \omega\}\).
 Suppose \(F \subseteq X\) is finite.
 Say \(|F| = m\), and write \(F = \{x_1,\cdots,x_m\}\).
 As \(\sigma_m\) is referenced infinitely many times throughout the play of this game and is winning for \(\plII\) on \(X^m\), there is an \(n \in \omega\) so that \((x_1,\cdots,x_m) \in (U_n)^m\).
 Then \(F \subseteq U_n\).
 Thus \(\{U_n : n \in \omega\}\) is an \(\omega\)-cover and \(\sigma\) is a winning strategy for \(\plII\).
 Therefore \(X\) is \(\Omega R^+\).

 (a) \(\Leftrightarrow\) (c) is a well-known result of Galvin\cite{Galvin}.

 (c) \(\Leftrightarrow\) (d) is 4.3 of Telgarksy\cite{Telgársky1975}.
  
 (e) \(\Rightarrow\) (d) is clear, but we want to show that (b) \(\Rightarrow\) (e).
 So assume \(X\) is \(\Omega R^+\). 
 Let \(\sigma\) be a winning strategy for \(\plII\) in \(\schStrongSelGame{\Omega_X}{\Omega_X}\). 
 To build a strategy \(\tau\) for \(\plI\) for \(\Omega FO(X)\), let \(s\in T(X)^{<\omega}\). 
 Assume \(\tau(t)\in[X]^{<\omega}\) has been defined for all \(t<s\), and \(\mc U_t\in\Omega_X\) is defined for all \(\emptyset<t\leq s\). 

 Suppose that for all \(F\in[X]^{<\omega}\), there existed \(U_F\in T(X)\) containing \(F\) such that for all \(\mc U\in\Omega_X\), \(U_F\not=\sigma(\<\mc U_{s\rest 1},\dots,U_s,\mc U\>)\). 
 Let \(\mc U=\{U_F:F\in[X]^{<\omega}\}\in\Omega_X\). 
 Then \(\sigma(\<\mc U_{s\rest 1},\dots,\mc U_s,\mc U\>)\) must equal some \(U_F\), demonstrating a contradiction.

 So there exists \(\tau(s)\in[X]^{<\omega}\) such that for all \(U\in T(X)\) containing \(\tau(s)\), there exists \(\mc U_{s\concat\<U\>}\in\Omega_X\) such that \(U=\sigma(\<\mc U_{s\rest 1},\dots,\mc U_s,\mc U_{s\concat\<U\>}\>)\). 
 (To complete the induction, \(\mc U_{s\concat\<U\>}\) may be chosen arbitrarily for all other \(U\in T(X)\).)

 So \(\tau\) is a strategy for \(\plI\) in \(\Omega FO(X)\). 
 Let \(\nu\) legally attack \(\tau\), so \(\tau(\nu\rest n)\subseteq \nu(n)\) for all \(n<\omega\). 
 It follows that \(\nu(n)=\sigma(\<\mc U_{\nu\rest 1},\dots,\mc U_{\nu\rest n},\mc U_{\nu\rest n+1}\>)\). 
 Since \(\<\mc U_{\nu\rest 1},\mc U_{n\rest 2},\dots\>\) is a legal attack against \(\sigma\), it follows that \(\{\sigma(\<\mc U_{\nu\rest 1},\dots,\mc U_{\nu\rest n+1}\>):n<\omega\}=\{\nu(n):n<\omega\}\) is an \(\omega\)-cover. 
 Therefore \(\tau\) is a winning strategy, verifying \(\plI\win\Omega FO(X)\).

 The equivalence of (c), (f), (h), and (i) are given as 3.8 of \cite{Tkachuk3}.

 The equivalence of (f) and (g) are given by Gruenhage \cite{Gruenhage1976}.
  
 The equivalence of (b), (j), and (k) are due to Clontz \cite{Clontz1}.

 (k) \(\Leftrightarrow\) (h) follows from 3.18a of \cite{Tkachuk3}, where
 Tkachuk refers to the \(sCDFT_p\) game as \(CLD(X,p)\).
\end{proof}

In \cite{Tkachuk2}, Tkachuk showed that for \(T_{3.5}\) spaces \(X\), \(X\) is uncountable if and only if \(C_p(X)\) is discretely selective.
We can rewrite this in terms of games using the following proposition.

\begin{proposition}
  For \(T_{3.5}\) spaces \(X\), \(X\) is uncountable if and only \(\plI\notprewin CD(C_p(X))\).
\end{proposition}

Combining this with several other results in the literature, we can see that the countability of \(X\) is equivalent to the existence of low information winning strategies for a variety of games.

\begin{theorem}
The following are equivalent for \(T_{3.5}\) spaces \(X\).
 \begin{enumerate}[a)]
  \item \(X\) is countable.
  \item \(\plII\markwin \schStrongSelGame{\mc O_X}{\mc O_X}\), that is, \(X\) is \(R^{+mark}\).
  \item \(\plII\markwin \schStrongSelGame{\Omega_X}{\Omega_X}\), that is, \(X\) is \(\Omega R^{+mark}\).
  \item \(\plI\prewin PO(X)\). 
  \item \(\plI\prewin FO(X)\).
  \item \(\plI\prewin\Omega FO(x)\).
  \item \(C_p(X)\) is first-countable.
  \item \(\plI\prewin \gruConGame{C_p(X)}{\mathbf 0}\).
  \item \(\plI\prewin \gruClusGame{C_p(X)}{\mathbf 0}\).
  \item \(\plI\prewin CL(C_p(X),\mathbf 0)\).
  \item \(\plI\prewin CD(C_p(X))\).
  \item \(\plII\markwin \schStrongSelGame{\Omega_{C_p(X),\mathbf 0}}{\Omega_{C_p(X),\mathbf 0}}\),
        that is, \(C_p(X)\) is \(sCFT^{+mark}\).
  \item \(\plII\markwin \schStrongSelGame{\mc D_{C_p(X)}}{\Omega_{C_p(X),\mathbf 0}}\),
        that is, \(C_p(X)\) is \(sCDFT^{+mark}\).
 \end{enumerate}
\end{theorem}

\begin{proof}
 (a) \(\Rightarrow\) (d) is straightforward. 
 So let \(\sigma\) be a predetermined strategy for \(\plI\) in \(PO(X)\). 
 If \(x\not\in\{\sigma(n):n<\omega\}\), let \(f(n)=X\setminus\{x\}\) for all \(n<\omega\). It follows that \(f\) is a legal counter-attack for \(\plII\) defeating \(\sigma\). 
 Thus not (a) implies not (d).

 We now prove that (b) is equivalent to (d).
 Let \(\sigma\) be a winning Markov strategy for \(\plII\) in \(\rothGame{X}\).
 Let \(n<\omega\). Suppose that for each \(x\in X\), there was an open neighborhood \(U_x\) of \(x\) where for every open cover \(\mc U\), \(\sigma(\mc U,n)\not=U_x\). 
 Then \(\sigma(\{U_x : x\in X\},n)\not\in\{U_x:x\in X\}\), a contradiction.

 So for each \(n<\omega\), there exists \(\tau(n)\in X\) such that for any open neighborhood \(U\) of \(\tau(n)\), there exists an open cover \(\mc U_n\) such that \(\sigma(\mc U_n,n)=U\). 
 Then \(\tau\) is a predetermined strategy for \(\plI\) in \(PO(X)\).

 It is also winning: for every attack \(f\) against \(\tau\), note that \(f(n)\) is an open neighborhood of \(\tau(n)\), so choose \(\mc U_n\) such that \(\sigma(\mc U_n,n)=f(n)\). 
 Then since \(\<\mc U_0,\mc U_1,\dots\>\) is a legal attack against \(\sigma\), it follows that \(\{f(n):n<\omega\}\) is an open cover of \(X\). 
 Therefore \(\tau\) is a winning predetermined strategy.
 So (b) implies (d).

 Now let \(\sigma\) be a winning predetermined strategy for \(\plI\) in \(PO(X)\). 
 For an open cover \(\mc U\) of \(X\) and \(n<\omega\), let \(\tau(\mc U,n)\) be any open set in \(\mc U\) containing \(\sigma(n)\). 
 It follows that \(\tau\) is a winning Markov strategy for \(\plII\) in \(\rothGame{X}\).
 Thus (d) implies (b).
 
 The previous paragraphs are easily modified to see that (c) is equivalent to (f).

 Clearly (d) implies (e), so we will see that (e) implies (a). 
 Let \(\sigma(n)\) be a predetermined strategy for I for \(FO(X)\). 
 Towards a contradiction, suppose that there is some \(x \in X \smallsetminus \bigcup_n \sigma(n)\).
 II could then play \(FO(X)\) as follows.
 At round \(n\) II can play an open set \(U_n\) which contains \(\sigma(n)\) but excludes \(x\).
 Then \(x \notin \bigcup_n U_n\), and so I has lost.
 This is a contradiction.
 So \(X = \bigcup_n \sigma(n)\), which means it is countable.

 It also clear that (f) implies (e), we will show that (a) implies (f).
 If \(X\) is countable, then so is \([X]^{<\omega}\), enumerate it as \(\{s_n : n \in \omega\}\).
 I's predetermined strategy for \(\Omega FO(X)\) is to play \(s_n\) are round \(n\).
 Clearly whatever II plays will be an \(\omega\)-cover.
 Thus (a) - (f) are equivalent.

 It is well-known and easy to see that (a) is equivalent to (g).

 To see that (g) implies (h), note that we can find a sequence of open sets \(U_n\) so that \(\mathbf 0 \in U_{n+1} \subseteq \overline{U_{n+1}} \subseteq U_n\) for all \(n\).
 I simply plays \(U_n\) at turn \(n\), and whatever \(x_n\) are played by II must converge to \(x\).

 Clearly (h) implies (j) which in turn implies (k), which is equivalent to (a) as noted before this theorem.

 (h) \(\Rightarrow\) (i) is evident; for the converse, let \(\tau(n)=\bigcap_{m\leq n}\sigma(m)\)
 where \(\sigma\) guarantees clustering. It follows that \(\tau\) guarantees that every subsequence
 clusters, and thus guarantees convergence.

 Clontz showed that (c), (l), and (m) are equivalent in \cite{Clontz1}.
 This completes the proof.

\end{proof}

In \cite{Tkachuk3}, Tkachuk characterizes \(\plII\win\Omega FO(X)\) as the second player having an ``almost winning strategy'' (\(\plII\) can prevent \(\plI\) from constructing an \(\omega\)-cover but perhaps not an arbitrary open cover) in \(PO(X)\), which he conflates with \(FO(X)\) as they are equivalent for ``completely'' winning perfect information strategies. 

But they cannot be interchanged in general.
\begin{proposition}
  Suppose \(X\) is \(T_1\).
  Then \(\plII\tactwin\Omega PO(X)\) if and only if \(|X| > 1\).
\end{proposition}
\begin{proof}
  First suppose that \(X=\{x\}\).
  Then \(\plI\) wins \(\Omega PO(X)\) by just playing \(x\) in round 1.
  So \(\plII\) does not have a winning tactic for \(\Omega PO(X)\).
  
  Now suppose that \(X\supseteq\{x_1,x_2\}\) for \(x_1\not=x_2\).
  Then let \(\sigma(x_1)=X\setminus\{x_2\}\), and \(\sigma(x)=X\setminus\{x_1\}\)
  otherwise. It follows that \(\{x_1,x_2\}\) is never contained in any
  set played by \(\sigma\), so \(\sigma\) never produces an \(\omega\)-cover,
  and thus is a winning tactic.
\end{proof}

However, if \(X\) is countable, then \(X\) is \(\Omega R^{+mark}\) and therefore \(\plI\prewin \Omega FO(X)\). 
So \(\Omega PO(X)\) is a very different game than those described previously.

\section{Strategies for player II for the Discrete Selection Game on \(C_p(X)\)}

Now we turn our attention to the opponent.
Our first observations hold for all spaces (not just \(T_{3.5}\) spaces or \(C_p(X)\)).
Consider the following games related to open covers.

\begin{proposition}
 The following are equivalent for all spaces \(X\).
 \begin{enumerate}[a)]
  \item \(\plII\win PO(X)\).
  \item \(\plII\markwin PO(X)\).
  \item \(\plII\win FO(X)\).
  \item \(\plII\markwin FO(X)\).
  \item \(\plI\win \rothGame{X}\).
  \item \(\plI\prewin \rothGame{X}\), that is, \(X\) is not \(R\).
 \end{enumerate}
\end{proposition}
\begin{proof}
 (a) \(\Leftrightarrow\) (c) is 4.4 of Telgarksy\cite{Telgársky1975}.

 The duality of \(PO(X)\) and \(\rothGame{X}\) for both players when considering perfect information is a well-known result of Galvin\cite{Galvin}. 
 So (a) is equivalent to (e).

 The equivalence of (e) and (f) is just a restatement of Pawlikowski's result that the Rothberger selection principle is equivalent to \(\plI\notwin \rothGame{X}\)\cite{Pawlikowski}, since the Rothberger selection principle is equivalent to \(\plI\notprewin\rothGame{X}\).

 We now prove that (f) and (b) are equivalent.
 Let \(\sigma\) be a winning predetermined strategy for \(\plI\) in \(\rothGame{X}\).
 For \(x\in X\) and \(n<\omega\), let \(\tau(x,n)\) be any open set in \(\sigma(n)\) containing \(x\). 
 It follows that \(\tau\) is a winning Markov strategy for \(\plII\) in \(PO(X)\).

 Now let \(\sigma\) be a winning Markov strategy for \(\plII\) in \(PO(X)\).
 We may defined the open cover \(\tau(n)=\{\sigma(x,n):x\in X\}\) of \(X\).
 It follows that \(\tau\) is a winning predetermined strategy for \(\plI\) in \(\rothGame{X}\).

 Finally, (d) implies (b) is obvious.
 We therefore finish the proof by showing that (b) implies (d).
 Let \(b:\omega^2\to\omega\) be a bijection.
 Given a winning Markov strategy \(\sigma\) for \(\plII\) in \(PO(X)\), define \(\tau(F_n,n)=\bigcup\{\sigma(x(i,n),b(i,n)):i<\omega\}\) where \(F_n=\{x(i,n):i<\omega\}\) (this indexing will cause at least one point to be repeated infinitely often, but this won't be a problem). 
 So given an attack \(\<F_0,F_1,\dots\>\) against \(\tau\), consider the attack \(g\) against \(\sigma\), where \(g(n)=x(m,k)\), where \(b(m,k) = n\). 
 It follows that
 \[
   X \not= \bigcup\{\sigma(g(n),n):n<\omega\} = \bigcup\{\sigma(x(i,n),b(i,n)):i,n<\omega\} = \bigcup\{\tau(F_n,n):n<\omega\}
 \]
 and therefore \(\tau\) is a winning Markov strategy for \(\plII\).
 Thus (b) implies (d).
\end{proof}

Similar results hold for games related to \(\omega\)-covers.

\begin{proposition}
The following are equivalent for all spaces \(X\).
 \begin{enumerate}[a)]
  \item \(\plII\win \Omega FO(X)\).
  \item \(\plII\markwin \Omega FO(X)\).
  \item \(\plI\win G_1(\Omega_X,\Omega_X)\).
  \item \(\plI\prewin G_1(\Omega_X,\Omega_X)\), that is, \(X\) is not \(\Omega R\).
 \end{enumerate}
\end{proposition}
\begin{proof}
 Let \(\sigma\) be a winning strategy for \(\plII\) in \(\Omega FO(X)\).
 For \(s\in([X]^{<\omega})^{<\omega}\), let \(\mc U_s=\{\sigma(s\concat\<F\>):F\in[X]^{<\omega}\}\). 
 Define the strategy \(\tau\) for \(\plI\) for \(G_1(\Omega_X,\Omega_X)\) recursively as follows.
 \begin{itemize}
     \item \(\tau\) opens with \(\mathcal{U}_\emptyset\). 
     That is \(\tau(\emptyset) = \mathcal{U}_\emptyset = \{\sigma(F) : F \in [X]^{<\omega}\}\).
     \item \(\plII\) must respond with some \(\sigma(F)\). \(\tau\) then plays \(\mathcal{U}_{<F>}\).
     \item At round \(n+1\), \(\plII\) will have just played some \(\sigma(F_0,\cdots,F_n)\). \(\tau\) will respond with \(\mathcal{U}_{<F_0,\cdots,F_n>}\).
 \end{itemize}
 This defines \(\tau\).
 Now suppose \(f\) is an attack by \(\plII\) against \(\tau\).
 \(f\) must look like \(\sigma(F_0),\sigma(F_0,F_1),\cdots\) for finite sets \(F_n \subseteq X\).
 As \(\sigma\) is winning for \(\plII\) in \(\Omega FO(X)\), it must be that \(\{\sigma(F_0),\sigma(F_0,F_1),\cdots\}\) is not an \(\omega\)-cover.
 So \(\tau\) is a winning strategy for \(\plI\) for \(G_1(\Omega_X,\Omega_X)\) and thus (a) implies (c).

 The equivalence of (c) and (d) is given by theorem 2 of \cite{Scheepers1997}.

 Let \(\sigma\) be a winning predetermined strategy for \(\plI\) in \(G_1(\Omega_x,\Omega_x)\). For \(F\in[X]^{<\omega}\) and \(n<\omega\), let \(\tau(F,n)\) be any open set in \(\sigma(n)\) containing \(F\). 
 It follows that \(\tau\) is a winning Markov strategy for \(\plII\) in \(\Omega FO(X)\), verifying that (d) implies (b).

 (b) implies (a) is trivial, so the proof is complete.
\end{proof}

\(\Omega R\) is equivalent to all finite powers being \(R\): see theorem 3 of \cite{Scheepers1997}. 
But \(\Omega R\) and \(R\) do not coincide in all models of \(ZFC\): see theorem 9 of \cite{BABINKOSTOVA2013} for a consistent example of a \(T_{3.5}\) \(R\) space \(X\) such that \(X^2\) is not \(R\), so therefore \(X\) is not \(\Omega R\). 
Note the distinction with strategies for the opponent, as \(R^+\) is equivalent to \(\Omega R^+\) and \(R^{+mark}\) is equivalent to \(\Omega R^{+mark}\).

Finally we will examine the point-picking games.

\begin{proposition}\label{lowerPropA}
The following properties imply lower properties for all spaces \(X\) and \(x\in X\).
 \begin{enumerate}[a)]
  \item \(\plI\win\schStrongSelGame{\mc D_X}{\Omega_{X,x}}\).
  \item \(\plII\win CL(X,x)\).
  \item \(\plII\win\gruClusGame{X}{x}\).
  \item \(\plI\win\schStrongSelGame{\Omega_{X,x}}{\Omega_{X,x}}\).
 \end{enumerate}
\end{proposition}
\begin{proof}
 Begin by letting \(\sigma\) be a winning strategy for \(\plI\) in \(\schStrongSelGame{\mc D_X}{\Omega_{X,x}}\). 
 For \(s\in T_{X}^{<\omega}\), assume \(\tau(s\rest i+1)\) is defined for \(i<|s|\), defining \(s'\in X^{|s|}\) by \(s'(i)=\tau(s\rest i+1)\), and let \(\tau(s\concat\<U\>)\in\sigma(s')\cap U\). 
 So \(\tau\) is a strategy for \(\plII\) in \(CL(X,x)\). Then for any attack \(f\) against \(\tau\), an attack \(f'\) against \(\sigma\) is defined by \(f'(i)=\tau(f\rest i+1)\). 
 It follows that \(\{f'(i):i<\omega\}=\{\tau(f\rest i+1):i<\omega\}\not\in\Omega_{X,x}\), so \(\tau\) is a winning strategy, witnessing (a) implies (b).

 Let \(\sigma\) be a winning strategy for \(\plII\) in \(CL(X,x)\). 
 Then \(\sigma\) is also a winning strategy for \(\plII\) in \(\gruClusGame{X}{x}\), so (b) implies (c).

 Given a winning strategy \(\sigma\) for \(\plII\) in \(\gruClusGame{X}{x}\), let \(s\in {T_{X,x}}^{<\omega}\) and suppose and \(B_t\in\Omega_{X,x}\) is defined for all \(t<s\). 
 Then let \(B_s=\{\sigma(s\concat\<U\>):U\in T_{X,x}\}\); it's clear that \(B_s\in\Omega_{X,x}\). 
 Define \(\tau\) for \(\plI\) in \(\schStrongSelGame{\Omega_{X,x}}{\Omega_{X,x}}\) by \(\tau(r)=B_{r'}\) where \(r'\in {T_{X,x}}^{|r|}\) satisfies \(r(i)=\sigma(r'\rest i+1)\) for all \(i<|r|\). 
 Then an attack \(f\) against \(\tau\) yields an attack \(f'\) against \(\sigma\) such that \(f(i)=\sigma(f'\rest i+1)\) for all \(i<\omega\). 
 Since \(\sigma\) is a winning strategy, it follows that \(\{f(i):i<\omega\}=\{\sigma(f'\rest i+1):i<\omega\}\not\in\Omega_{X,x}\). 
 This verifies (c) implies (d).
\end{proof}

\begin{proposition}\label{lowerPropB}
 The following properties imply lower properties for all spaces \(X\) and \(x\in X\).
 \begin{enumerate}[a)]
  \item \(\plI\prewin\schStrongSelGame{\mc D_X}{\Omega_{X,x}}\).
  \item \(\plII\markwin CL(X,x)\).
  \item \(\plII\markwin\gruClusGame{X}{x}\).
  \item \(\plI\prewin\schStrongSelGame{\Omega_{X,x}}{\Omega_{X,x}}\).
 \end{enumerate}
\end{proposition}
\begin{proof}
 Begin by letting \(\sigma\) be a winning predetermined strategy for \(\plI\) in \(\schStrongSelGame{\mc D_X}{\Omega_{X,x}}\). 
 Define the Markov strategy \(\tau\) for \(\plII\) in \(CL(X,x)\) by choosing \(\tau(U,n)\in\sigma(n)\cap U\). 
 Since \(\tau(U,n)\in\sigma(n)\) for all \(n<\omega\), it's clear that \(\{\tau(U,n):n<\omega\}\not\in\Omega_{X,x}\), making \(\tau\) a winning strategy, witnessing (a) implies (b).

 Let \(\sigma\) be a winning Markov strategy for \(\plII\) in \(CL(X,x)\). Then \(\sigma\) is also a winning Markov strategy for \(\plII\) in \(\gruClusGame{X}{x}\), so (b) implies (c).

 Given a winning Markov strategy \(\sigma\) for \(\plII\) in \(\gruClusGame{X}{x}\), let \(\tau(n)=\{\sigma(U,n):U\in T_{X,x}\}\). 
 Then \(\tau\) is a predetermined strategy for \(\plI\) in \(\schStrongSelGame{\Omega_{X,x}}{\Omega_{X,x}}\). 
 For any attack \(f\) against \(\tau\), \(f(n)=\sigma(g(n),n)\) for some \(g(n)\in T_{X,x}\). 
 But then \(g\) is an attack against \(\sigma\), and thus \(\{f(n):n<\omega\}=\{\sigma(g(n),n):n<\omega\}\not\in\Omega_{X,x}\), so we have (c) implies (d).
\end{proof}

We will see in the upcoming theorem that for \(C_p(X)\) with \(X\) \(T_{3.5}\), (a)-(d) in both of the previous propositions are actually equivalent.

\begin{theorem}
 The following are equivalent for all \(T_{3.5}\) spaces.
 \begin{enumerate}[a)]
  \item \(\plII\win \Omega FO(X)\).
  \item \(\plII\markwin \Omega FO(X)\).
  \item \(\plI\win G_1(\Omega_X,\Omega_X)\).
  \item \(X\) is not \(\Omega R\), that is, \(\plI\prewin G_1(\Omega_X,\Omega_X)\).
  \item \(\plI\win\schStrongSelGame{\Omega_{C_p(X),\mathbf 0}}{\Omega_{C_p(X),\mathbf 0}}\).
  \item \(C_p(X)\) is not \(sCFT\), that is, \(\plI\prewin \schStrongSelGame{\Omega_{C_p(X),\mathbf 0}}{\Omega_{C_p(X),\mathbf 0}}\).
  \item \(\plI\win\schStrongSelGame{\mc D_{C_p(X)}}{\Omega_{C_p(X),\mathbf 0}}\).
  \item \(C_p(X)\) is not \(sCDFT\), that is, \(\plI\prewin \schStrongSelGame{\mc D_{C_p(X)}}{\Omega_{C_p(X),\mathbf 0}}\).
  \item \(\plII\win\gruClusGame{C_p(X)}{\mathbf 0}\).
  \item \(\plII\markwin\gruClusGame{C_p(X)}{\mathbf 0}\).
  \item \(\plII\win CL(C_p(X),\mathbf 0)\).
  \item \(\plII\markwin CL(C_p(X),\mathbf 0)\).
  \item \(\plII\win CD(C_p(X))\).
  \item \(\plII\markwin CD(C_p(X))\).
 \end{enumerate}
\end{theorem}
\begin{proof}
 (a)-(d) were shown in Proposition 19.
 The equivalence of (d), (f), and (h) was shown by Sakai\cite{Sakai}.
 The equivalence of (f) and (e) is given in 4.37 of \cite{CombOpenCovers}.

 Of course (h) implies (g). And since \(\mc D_{C_p(X)}\subseteq\Omega_{C_p(X),\mathbf 0}\), any winning strategy for \(\plI\) in \(\schStrongSelGame{\mc D_{C_p(X)}}{\Omega_{C_p(X),\mathbf 0}}\) is a winning strategy for \(\plI\) in \(\schStrongSelGame{\Omega_{C_p(X),\mathbf 0}}{\Omega_{C_p(X),\mathbf 0}}\), so (g) implies (e). 
 We have so far shown that (a) - (h) are equivalent.

 Proposition 20 established that (g) \(\Rightarrow\) (k) \(\Rightarrow\) (i) \(\Rightarrow\) (e).
 We just proved, however, that (g) and (e) are equivalent.
 So (e), (g), (i), and (k) are equivalent.
 Proposition 21 established that (h) \(\Rightarrow\) (l) \(\Rightarrow\) (j) \(\Rightarrow\) (f).
 Again, we just saw that (f) and (h) are equivalent.
 So (f), (h), (j), and (l) are equivalent.
 Thus (a) - (l) are equivalent.

 Assuming (b), we adapt Proposition 3.9 of \cite{Tkachuk3} as follows.
 Let \(\sigma\) be a winning Markov strategy for \(\plII\) in \(\Omega FO(X)\). 
 Then for \(U=[\mathbf x(U),supp(U),\epsilon(U)]\in T_{C_p(X)}\), let \(\tau(U,n)\in C_p(X)\) satisfy  \(\tau(U,n)(x)=\mathbf x(U)(x)\) for \(x\in F\) and \(\tau(U,n)(x)=n\) for \(x\in X\setminus\sigma(U,n)\). 
 Then \(\tau\) is a Markov strategy for \(\plII\), and when it is attacked by \(f\), we note that \(\{\sigma(supp(f(n)),n):n<\omega\}\) is not an \(\omega\)-cover. 
 So choose \(G\in[X]^{<\omega}\) such that \(G\not\subseteq\sigma(supp(f(n)),n)\) for all \(n<\omega\). 
 Then for \(\mathbf y\in C_p(X)\), choose \(m\) such that \(\mathbf y(x)<m\) for all \(x\in G\). 
 Note then that for \(n\geq m\), there exists \(x\in G\setminus\sigma(f(n),n)\) such that \(\tau(f(n),n)(x)=n\geq m\). 
 Then \(\{\mathbf z\in C_p(X):\mathbf z(x)<m\text{ for all }x\in G\}\) is an open neighborhood of \(\mathbf y\) that misses \(\tau(f(n),n)\) for all \(n\geq m\), so it follows that \(\{\tau(f(n),n):n<\omega\}\) is closed and discrete in \(C_p(X)\). 
 Therefore \(\tau\) is a winning Markov strategy, verifying (b) implies (n).

 It's clear that (n) implies (m), so finally note that a winning strategy for \(\plII\) in \(CD(C_p(X))\) is also a winning strategy for \(\plII\) in \(CL(C_p(X),\mathbf 0)\), so (m) implies (k).
 This completes the equivalence.
\end{proof}

The equivalence of (a) and (m) answers Question 4.6 of Tkachuk in \cite{Tkachuk3}.

\section{Open Problems}

\begin{question}
  In \cite{Tkachuk2}, Tkachuk found sufficient conditions for \(C_p(X,\mathbb{I})\) 
  to satisfy the discrete selection princple. 
  What happens when we play the discrete selection game on \(C_p(X,\mathbb I)\)?
\end{question}

\begin{question}
 Is there a point-picking game on \(C_p(X)\) which characterizes when \(X\) is not \(R\)?
\end{question}

\begin{question}
 There is a model of \(ZFC\) where \(R\) and \(\Omega R\) are distinct properties.
 Is it consistent that they are the same?
 That is, is there a universe of ZFC in which every \(R\) space is also \(\Omega R\)?
\end{question}

\begin{question}
 All the games played in this paper had length \(\omega\).
 Do these equivalences continue to hold for longer games?
\end{question}

\begin{question}
 The implications in Propositions \ref{lowerPropA} and \ref{lowerPropB} reverse
 when \(X=C_p(Y)\) for some \(T_{3.5}\) space \(Y\). When in general can these
 implications reverse?
\end{question}


\bibliographystyle{elsarticle-num}
\bibliography{bibliography}

\end{document}